\documentclass[12pt, twoside, leqno]{article}

\usepackage{amsmath,amsthm}
\usepackage{amssymb}

\usepackage{enumerate}

\usepackage{graphicx}

\newtheorem{thm}{Theorem}[section]
\newtheorem{cor}[thm]{Corollary}

\newtheorem{prop}[thm]{Proposition}

\theoremstyle{definition}
\newtheorem{defin}[thm]{Definition}

\numberwithin{equation}{section}

\usepackage{pgf,tikz}
\usetikzlibrary{arrows}

\begin{document}

\baselineskip=12.1pt


\title{ The spectrum of a class of graphs derived from Grassmann graphs}

\author{S.Morteza Mirafzal*, Roya Kogani\\
Department of Mathematics \\
  Lorestan University, Khorramabad, Iran\\
E-mail: mirafzal.m@lu.ac.ir\\
E-mail: smortezamirafzal@yahoo.com\\
E-mail: Rkogani@yahoo.com}

\date{}

\maketitle

\renewcommand{\thefootnote}{}

\footnote{2010 \emph{Mathematics Subject Classification}: 05C50, 05C25 
}
\footnote{\emph{Keywords}:Grassmann graph, connected graph, spectrum.}

\footnote{\emph{*Coresponding author}: }

\footnote{\emph{Date}:  }

\renewcommand{\thefootnote}{\arabic{footnote}}
\setcounter{footnote}{0}
\date{}

\begin{abstract} Let $n,k$ be positive integers such that $n\geq 3$, $k < \frac {n}{2} $. Let $q$ be a power of a prime $p$ and $\mathbb{F}_q$ be a finite field of order $q$. Let $V(q,n)$ be a vector space of dimension $n$ over $\mathbb{F}_q$. We define the graph $S(q,n,k)$ as  a  graph with the vertex set $V=V_k \cup V_{k+1}$, where $V_k$ and $V_{k+1}$ are the family of subspaces in $V(q,n)$ of dimension $k$ and $k+1$ respectively,     in which two vertices $v$ and $w$ are adjacent whenever $v$ is a subspace of $w$  or $w$ is a subspace of $v$.
   It is clear that the graph $S(q,n,k)$ is a bipartite graph. In this paper, we study some properties of this graph. In particular,  we determine the spectrum of the graph $S(q,n,k)$.

\end{abstract}

\maketitle
\section{ Introduction and Preliminaries }  In this paper, a graph $\Gamma=(V,E)$ is
considered as an undirected simple graph where $V=V(\Gamma)$ is the vertex-set
and $E=E(\Gamma)$ is the edge-set. For all the terminology and notation
not defined here, we follow [2,3,4,5,6].\

A graph is distance-transitive if, for any two arbitrarily-chosen pairs of vertices at
the same distance, there is some automorphism of the graph taking the first pair
onto the second. Distance-transitive graphs have various nice properties which have been investigated by many researchers in abstract and applied graph theory.   Cycles, hypercubes, Johnson graphs, Hamming graphs and Grassmann graphs are the most known examples of the class of  distance-transitive graphs [4]. Hamming graphs, which hypercubes are a subclass of them, are also   important in coding theory and its applications [1].   In algebraic graph theory,  when we work on a family of graphs, we encounter  two main problems. The first problem is determining their  automorphism groups and the second problem is determining  their spectrums. Recently, the automorphism groups of the Hamming graphs are determined by a new and short proof [9]. From a given graph, we can construct some other graphs which may have some good desired properties. In this paper, we wish to determine the spectrum of a class of graphs which can be constructed and studied by Grassmann graphs. \

 Let $p$ be a positive prime integer and $q=p^m$ where $m$ is a positive integer. Let $n,k$ be positive integers and  $k <n$.  Let $V(q,n)$ be a vector space of dimension $n$ over the finite field $\mathbb{F}_q.$                                                                                                                                        Let $V_k$ be the family  of all
subspaces  of  $V(q,n)$ of dimension $k$.
Every element of $V_k$ is also called a $k$-subspace( or $k$-space ).
The Grassmann graph $G(q,n,k)$ is the graph with the  vertex-set $V_{k}$,  in which two vertices $u$ and $w$ are adjacent if and only if $\dim(u\cap w)=k-1$.\

Note that if $k = 1$, we have a complete graph, so we shall assume that $k >1 $.
 It is clear that  the number of vertices of the Grassmann graph $G(q,n,k)$, that is,
  $|V_k|$  is
the Gaussian binomial coefficient,
$${n\brack k}_q={n\brack k}= \dfrac{(q^{n}-1)(q^n-q)\cdots (q^{n}-q^{k-1})}{(q^{k}-1)(q^k-q)\cdots (q^k-q^{k-1})} =\dfrac{(q^{n}-1)\cdots (q^{n-k+1}-1)}{(q^{k}-1)\cdots (q-1)}.$$

Noting that ${n\brack k}_q={n\brack n-k}_q$,  it follows that $|V_k|=|V_{n-k}|$. It is easy to show that if $1 \leq i < j \leq \frac{n}{2}$, then $|V_i| < |V_j|$.  It can be shown that
    $G(q,n,k) \cong G(q,n,n-k)$ [3],  and  hence in the sequel we assume that $k \leq\frac{n}{2}$.

It is easy to see that  the distance between two vertices $v$ and $w$ in the graph $G(q,n,n-k)$   is $k-dim(v\cap w)$.
The Grassmann graph is a distance-regular graph of diameter $k$ [4].  Concerning the graph $G(q,n,n-k)$ we have the following fact.
\begin{thm}
 Let $V=V(q,n)$ be as above. Suppose $0 \leq i,j \le n$. Then \\
(i)  \ The number of $k$-spaces of $V$ is ${n\brack k}_q={n\brack k}$\\
(ii)  \ If $X$ is a $j$-space of $V$, then there are
precisely $q^{ij}{n-j \brack i}$ $i$-spaces $Y$ in $V $ such that $X \cap Y = 0$. \\
(iii)  \ If $X$ is a $j$-space of $V$, then there are
precisely $q^{(i-m)(j-m)}{n-j \brack i-m}{j \brack m}$ $i$-spaces $Y$ in $V $ such that $X \cap Y$ is an $m$-space.

\end{thm}
Now, we can deduce the following fact from $(iii)$ of Theorem 1.1.
\begin{cor}If $X$ is a $k$-space of   $V$, then there are
precisely\\ $q^{(k+1-k)(k-k)}{n-k \brack k+1-k}{k \brack k}$=$n-k \brack 1$, $(k+1)$-spaces $Y$ in $V $ such that $X \leq  Y$.

\end{cor}

\begin{proof} We know  that a $(k+1)$-space $Y$ contains the $k$-space $X$ if and only if $Y \cap X$ is the $k$-space $X$. Now the result follows from Theorem 1.1.

\end{proof}

\begin{defin}
Let $n,k$ be positive integers such that $n\geq 3$, $k<n$. Let $q$ be a power of a prime $p$ and $\mathbb{F}_q$ be a finite field of order $q$. Let $V(q,n)$ be a vector space of dimension $n$ over $\mathbb{F}_q$. We define the graph $S(q,n,k)$ as  a  graph with the vertex-set $V=V_k \cup V_{k+1}$, where $V_k$ and $V_{k+1}$ are the family of subspaces in $V(q,n)$ of dimension $k$ and $k+1$ respectively,     in which two vertices $v$ and $w$ are adjacent whenever $v$ is a subspace of $w$  or $w$ is a subspace of $v$.
\end{defin}

    It is clear that this graph  is a bipartite graph with partition  $V=V_k\cup V_{k+1}$,
$$V_k=\{v\subset V(q,n) | dim(v)=k\},$$
$$V_{k+1}=\{v\subset  V(q,n) |dim(v)=k+1\}.$$

When $n=2k+1$, the  graph $S(q,n,k)$ is called a doubled Grassmann graph and some of its properties have been studied by some authors [4,7]. The automorphism group of the graph $S(q,n,k)$ has recently been determined [8]. We know that the number of $k$-subspaces of a vector space $V(q,n)$   is the Gaussian binomial coefficient
${n\brack k}_q={n\brack k}$.
Thus, $|V_k|={n \brack k}$ and $|V_{k+1}|={n\brack k+1}$,   hence the order of  the graph $S(q,n,k)$ is ${n \brack k}+{n \brack k+1}$. If $v\in V_k$, then from Corollary 1.2, it follows that   $deg(v)={n-k\brack 1}=\frac{q^{n-k}-1}{q-1}=q^{n-k-1}+q^{n-k-2}+\dots+q+1$. We know that the number of $k$-subspaces of a vector space of dimension $k+1$ is the Gaussian binomial coefficient ${k+1\brack k}$, hence if $v\in V_{k+1}$, then   $deg(v)={k+1\brack k}={k+1\brack 1}$=$\frac{q^{k+1}-1}{q-1}$.
When a bipartite graph $G=(P,E)$ is such that  $P=P_1 \cup P_2$, $P_1 \cap P_2= \emptyset $  $|P_1|=n_1, |P_2|=n_2$,  also each vertex in $V_1$ is of degree $r_1$ and each vertex in $P_2$ is of degree $r_2$, then the graph $G$ is called a bi-regular bipartite graph with parameters $(n_1,n_2,r_1,r_2)$.
  Thus the graph $S(q,n,k)$ is a  bi-regular  bipartite graph with parameters $({n\brack k} ,{n\brack k+1} ,{n-k\brack 1} , {k+1\brack 1})$. It can be shown that the graph $S(q,n,k)$ is  connected  (Proposition 2.1). We can easily see that this graph  is a regular graph, when  $n=3, k=1$,  since $r_1=r_2=q+1$ and in this case $|V_1|=|V_2|=q^2+q+1$.
   Noting that ${n\brack k}={n\brack n-k}$, it is easy to show that $S(q,n,k) \cong S(q,n,n-k-1)$, hence in the sequel we assume $k < \frac{n}{2}$. \

Let $\Gamma$ be a graph with vertex set $V=\{v_1,v_2,\dots,v_n\}$ and edge set $E(\Gamma)$. The adjacency matrix $A=A(\Gamma)=[a_{ij}]$ of $\Gamma$ is an $n\times n$ symmetric matrix of $0$'s and $1$'s with $a_{ij}=1$ if and only if $v_i$ and $v_j$ are adjacent. The characteristic polynomial of $\Gamma$ is the polynomial $P(G)=P(G,x)=det(xI_n-A)$, where $I_n$ denotes the $n\times n$ identity matrix. The spectrum of $A(\Gamma)$ is also called the spectrum of $\Gamma$. If the eigenvalue of $\Gamma$ are ordered by $\lambda_1>\lambda_2>\dots>\lambda_r$, and their multiplicities are $m_1$,$m_2$,$\dots$,$m_r$, respectively, then we write ;
\begin{equation*}
Spec(\Gamma)=\begin{pmatrix}
\lambda_1,\lambda_2,\dots,\lambda_r\\
m_1,m_2,\dots,m_r
\end{pmatrix}
\hspace{0.5 cm}
or
\hspace{0.5 cm}
Spec(\Gamma)=\big\lbrace\lambda_1^{m_1},\lambda_2^{m_2},\dots,\lambda_r^{m_r}\big\rbrace
\end{equation*}

Let $\mathbb{F}_q$ be a finite field of order $q$ and let $V(q,n)$ be an  $n$-dimensional vector space over the field $\mathbb{F}_q$. The Grassmann graph $G(q,n,k)$  is a graph whose vertex-set is    the family of  $k$-subspaces of $V(q,n)$, in which two vertices $v$ and $w$   are adjacent if and only if $dim(v\cap w)=k-1$. Concerning the spectrum of the Grassmann graph   $G(q,n,k)$, we have the following fact [3].
\begin{thm}\label{t1}
Let $\Gamma$ be the Grassmann graph $G(n,q,k)$.  Then $\Gamma$ has diameter $d=min(k,n-k)$. Moreover,  $\Gamma$ has eigenvalues and multiplicities given by
$$\theta_j=q^{j+1}{k-j\brack 1}{n-k-j\brack1}-{j\brack 1}$$
and
$$f_j={n\brack j}-{n\brack j-1}$$
where $0\leq j\leq d$.
\end{thm}
In this paper, we wish to  determine the spectrum of the graph $S(q,n,k)$.
\section{Main Results}
In this section, we study some properties of the graph $\Gamma=S(q,n,k)$. In the first step, we show that $\Gamma$ is a connected graph.

\begin{prop}\label{p1}
The graph $\Gamma=S(q,n,k)$  is a  connected  graph.
                                                                                                                                                                                                                       \end {prop}

                                                                                                                                      \begin{proof} It is clear that the graph $G=S(q,n,k)$ is a bipartite graph with partition $V_k \cup V_{k+1}$.  We now show that $G$ is a connected graph. It is sufficient to show that if $v_1,v_2$ are two vertices in $V_{k}$, then there is a path in $G$ between $v_1$ and $v_2$. Let $dim(v_1 \cap v_2)=k-j$, $1 \leq j \leq k$.
                                                                                                                                      We prove our assertion by induction on $j$. If $j=1$, then $u=v_1+v_2$ is a subspace of $V(q,n)$ of dimension $k+k-(k-1)=k+1$, which contains both of $v_1$ and $v_2$.   Hence,  $u \in V_{k+1}$ is adjacent to both of the vertices $v_1$ and $v_2$.
                                                                                                                                       Thus, if $j=1$, then there is a path between $v_1$ and $v_2$ in the graph $G$. Assume when $j=i$,  $0 < i <k$, then there is a path in $G$ between $v_1$ and $v_2$. We now assume $j=i+1$. Let $v_1 \cap v_2=w$, and let $B=\{ b_1,...,b_{k-i-1}  \}$ be a basis for the subspace $w$ in the space $V(q,n)$. We can extend $B$ to   bases $B_1$ and $B_2$ for the subspaces
                                                                                                                                        $v_1$ and $v_2$, respectively. Let $B_1= \{ b_1,...,b_{k-i-1}, c_1,...,c_{i+1}  \}$ be a basis for $v_1$ and $B_2= \{ b_1,...,b_{k-i-1}, d_1,...,d_{i+1}  \}$ be a basis for $v_2$. Consider the subspace $s=<b_1,...,b_{k-i-1}, c_1,d_2,...,d_{i+1}>$.
                                                                                                                                         Then $s$ is a $k$-subspace of the space $V(q,n)$ such that $dim(s \cap v_2)=k-1$ and $dim(s \cap v_1)=k-i$. Hence by the induction assumption, there is a path $P_1$ between vertices $v_2$ and $s$, and a path $P_2$ between vertices $s$ and $v_1$. We now conclude that there is a path in the graph $G$ between vertices  $v_1$ and $v_2$.
                                                                                                                                         \end{proof}

We now proceed to determine the spectrum of the graph $S(q,n,k)$.

\begin{thm}
Let $V(q,n)$ be a vector space of dimension $n$ over a field $\mathbb{F}_q$, where $q$ is a power  of a prime $p$. Let $k < \frac{n}{2}$ and  $\Gamma=S(q,n,k)$. Then, the graph $\Gamma$ has distinct eigenvalues $ \pm \frac{1}{(q-1)} \sqrt{\theta_j}$, where
$$\theta_j=(q^{(n-j+1)}-q^{(n-k)}+q^j-q^{(k+1)}), $$
with the multiplicity  $f_j$,
$$f_j=2({n\brack j}-{n\brack j-1}), \    0\leq j\leq k,$$
and $$\theta_{k+1}=0$$
with   the multiplicity
$$f_{k+1}={n\brack k+1}-{n\brack k}.$$

\end{thm}
\begin{proof}
Let $A$ be the adjacency matrix of the graph $\Gamma$. In the first step, we determine the spectrum of the matrix $A^2$. Let $(A^2)_{v,w}$ be the entry in the $v$th row and $w$th column of $A^2$. Note that $(A^2)_{v,w}$ is the number of $2$-paths between the vertices $v$ and $w$. We have the following cases.\\
(i) If $v=w$, then $(A^2)_{v,w}$ is the number of neighbors of $v$, hence we have;
\[
(A^2)_{v,v}=
\begin{cases}
{n-k\brack 1},  \  if \ v\in V_k\\
\hspace{0.2 cm}{k+1\brack k},   \  if \ v\in V_{k+1}

\end{cases}
\]
(ii) Let $v\neq w$ and $v\in V_k$ and $w\in V_{k+1}$  (or vice versa). Thus,  there is not $2$-path between $v$ and $w$ since the graph $\Gamma$ is a bipartite graph.  Indeed, we have $(A^2)_{v,w}=0$. \\
(iii) Let $v\neq w$, $v,w\in V_k $ and  $P:vuw$ be a $2$-path in the graph $\Gamma$. Thus $u\in V_{k+1}$ is a $(k+1)$-subspace of $V(q,n)$ such that it contains both of $k$-subspaces $v$ and $w$.  Thus $u$ must contain the subspace $v+w$. From the fact that $dim(v+w)=dim(v)+dim(w)-dim(u \cap w)$,  it follows that $dim(v \cap w)=k-1$, and $u=v+w$. In other words,  there is exactly 1 path of length 2 between $v$ and $w$ if and only if $dim(v \cap w)=k-1$.    We now, consider $v,w\in V_k$ as vertices of Grassmann graph $G(q,n,k)$. Let $G_k$ be the adjacency matrix of $G(q,n,k)$, then
\begin{center}
$ (A^2)_{v,w}=1$ if and only if  $(G_k)_{v,w}=1$, and $ (A^2)_{v,w}=0$ if and only if  $(G_k)_{v,w}=0$
\end{center}
where $v,w\in V_k$ and $v\neq w$.\\
(iv) We now, consider $v,w\in V_{k+1}$ as vertices of Grassman graph $G(q,n,k+1)$ with vertex set $V_{k+1}$. By a similar argument which we saw in (iii), we have the following fact,\\
there is a path of length two (in the graph) between $v,w\in V_{k+1}$ if and only if $dim(v\cap w)=k$.\\
Therefore, there is a $2$-path between $v,w\in V_{k+1}$ as vertices of $\Gamma$ if and only if $v,w$ are adjacent as vertices of $G(q,n,k+1)$. Let $G_{k+1}$ be the adjacency matrix of the Grassmann graph $G(q,n,k+1)$, then
\begin{center}
$ (A^2)_{v,w}=1$ if and only if $ (G_{k+1})_{v,w}=1$, and $ (A^2)_{v,w}=0$ if and only if $ (G_{k+1})_{v,w}=0$.
\end{center}
where $v,w\in V_{k+1}$ and $v\neq w$.\\
By our discussion we deduce that,
\[
A^2=
\begin{bmatrix}
{n-k\brack 1}I_r+G_k & 0\\
0 & {k+1\brack k}I_s+G_{k+1}
\end{bmatrix}
\]

where $r={n\brack k}$ and $s={n\brack k+1}$. We now can determine the characteristic polynomial of the matrix $A^2$.
$$P(A^2)=det(\lambda I-A^2)=det(\lambda I_r-{n-k\brack 1}I_r -G_k)det(\lambda I_s-{k+1\brack k}I_s - G_{k+1}).$$
 We know from Theorem 1.4,  the spectrum of  the Grassmann graph $G(q,n,k)$.
 We now deduce  that  the eigenvalues of the matrix  $A^2$ are,\\
$$\theta_j=q^{j+1}{k-j\brack 1}{n-k-j\brack1}-{j\brack 1}+{n-k\brack 1}$$
with the multiplicity
$$f_j={n\brack j}-{n\brack j-1},  \ 0\leq j\leq k$$

and \\
$$\gamma_i=q^{i+1}{k+1-i\brack 1}{n-k-1-i\brack1}-{i\brack 1}+{k+1\brack k}$$
with the multiplicity
$$e_i={n\brack i}-{n\brack i-1}, \ 0\leq i\leq k+1. $$
On the other hand, we have,
$$\theta_j= q^{j+1}{k-j\brack 1}{n-k-j\brack1}-{j\brack 1}+{n-k\brack 1}=$$$$q^{j+1}(\frac{q^{k-j}-1}{q-1}\times \frac{q^{n-k-j}-1}{q-1})-\frac{q^j-1}{q-1}+\frac{q^{n-k}-1}{q-1})=$$
$$ \frac{1}{(q-1)^2}(q^{j+1}(q^{k-j}-1)(q^{n-k-j}-1)-(q^j-1)(q-1)+(q^{n-k}-1)(q-1))=$$
$$\frac{1}{(q-1)^2} (q^{j+1}(q^{n-2j}-q^{k-j}-q^{n-k-j}+1)-q^{j+1}+q^j+q-1+q^{n-k+1}-q^{n-k}-q+1)= $$
$$\frac{1}{(q-1)^2}(q^{n-j+1}-q^{k+1}-q^{n-k+1}+q^{j+1}-q^{j+1}+q^j+q-1+q^{n-k+1}-q^{n-k}-q+1)=    $$
$$\frac{1}{(q-1)^2} (q^{n-j+1}-q^{k+1}+q^j-q^{n-k}).$$
Also, we have
$$\gamma_i= q^{i+1}{k+1-i\brack 1}{n-k-1-i\brack1}-{i\brack 1}+{k+1\brack k}=$$
$$q^{i+1}(\frac{q^{k+1-i}-1}{q-1}\times\frac{q^{n-k-1-i}-1}{q-1})-\frac{q^i-1}{q-1}+\frac{q^{k+1}-1}{q-1}=$$
$$ \frac{1}{(q-1)^2}(q^{i+1}(q^{n-2i}-q^{k+1-i}-q^{n-k-1-i}+1)-q^{i+1}+q^i+q-1+q^{k+2}-q^{k+1}-q+1)=  $$
$$\frac{1}{(q-1)^2}(q^{n-i+1}-q^{k+2}-q^{n-k}+q^{i+1}-q^{i+1}+q^i+q-1+q^{k+2}-q^{k+1}-q+1)=$$
$$\frac{1}{(q-1)^2} (q^{n-i+1}-q^{n-k}+q^i-q^{k+1}).  $$
Now, it is easy to check that for $0\leq i,j\leq k$, we have $\theta_j=\gamma_i$ if and only if $i=j$. Also,  if $i=k+1$, then $\gamma_i=\frac{1}{(q-1)^2} (q^{n-k-1+1}-q^{n-k}+q^{k+1}-q^{k+1})=0. $ Since the eigenvalues of $A^2$ are squares of the eigenvalues of $A$ and since $S(q,n,k)$ is a bipartite graph, then each eigenvalue  of $\Gamma=S(q,n,k)$ is of the form, \\

$\pm\frac{1}{(q-1)} \sqrt{\theta_j}, \  $with multiplicity$ \  f_j=2({n\brack j}-{n\brack j-1}), 0 \leq j \leq k, $ \\ and  \\

$\gamma_{k+1}=0,  $ with multiplicity$ \ f_{k+1}={n\brack k+1}-{n\brack k}. $

\end{proof}

 \

\end{document}